\documentclass[12pt]{article} 

\usepackage{amsmath}
\usepackage{amsthm}
\usepackage{amsfonts}
\usepackage{setspace}
\usepackage{fullpage}
\usepackage{amssymb}
\usepackage{enumitem}
\usepackage{bbold} 
\usepackage{comment}
\usepackage{pgf,tikz}
\usepackage[colorlinks=true,citecolor=black,linkcolor=black,urlcolor=blue]{hyperref}
\usetikzlibrary{arrows}

\newtheorem{theorem}{Theorem} 
\newtheorem{lemma}[theorem]{Lemma}
\newtheorem{corollary}[theorem]{Corollary}

\newcommand{\A}{\mathcal{A}}

\newcommand{\C}{\mathcal{C}}

\newcommand{\p}{\mathcal{P}}
\newcommand{\E}{\mathcal{C'}}

\newcommand{\R}{\mathcal{R}}

\newcommand{\abs}[1]{\left\lvert{#1}\right\rvert}
\newcommand{\floor}[1]{\left\lfloor{#1}\right\rfloor}
\newcommand{\ceil}[1]{\left\lceil{#1}\right\rceil}

\DeclareMathOperator{\La}{La}

\title{Exact forbidden subposet results using chain decompositions of the cycle}
\author{Abhishek Methuku\\
\small Department of Mathematics\\[-0.8ex]
\small Central European University\\[-0.8ex]
\small Budapest, Hungary\\
\small\tt abhishekmethuku@gmail.com\\
\and
Casey Tompkins\\
\small Department of Mathematics \\[-0.8ex]
\small Central European University\\[-0.8ex]
\small R\'enyi Institute, Hungarian Academy of Sciences\\[-0.8ex]
\small Budapest, Hungary\\
\small\tt ctompkins496@gmail.com\\
}

\begin{document}

\maketitle

\begin{abstract}
We introduce a method of decomposing the family of intervals along a cyclic permutation into chains to determine the size of the largest family of subsets of $[n]$ not containing one or more given posets as a subposet.   De Bonis, Katona and Swanepoel determined the size of the largest butterfly-free family. We strengthen this result by showing that, for certain posets containing the butterfly poset as a subposet, the same bound holds. We also obtain the corresponding LYM-type inequalities.

\end{abstract}

\section{Introduction}
Let $[n] = \{1,2,\dots,n\}$ and $2^{[n]}$ be the power set of $[n]$.  Let $\binom{[n]}{r}$ denote the collection of all $r$-element subsets of $[n]$.    The sum of the $k$ largest binomial coefficients of the form $\binom{n}{i}$ is denoted by $\Sigma(n,k)$. A collection of sets, $\mathcal{A} \subset 2^{[n]}$, may be viewed as a partially ordered set (poset) with respect to inclusion.

Given two posets, $P$ and $Q$, we say that $P$ is a (weak) subposet of $Q$ if there is an injection, $\phi$, from $P$ to $Q$ such that $x \le y$ in $P$ implies $\phi(x) \le \phi(y)$ in $Q$.  If we also have that $\phi(x) \le \phi(y)$ implies $x \le y$, then $P$ is called an induced subposet of $Q$. 

For a collection of posets, $\p$, Katona and Tarjan \cite{katona1983extremal} introduced the function $\La(n,\p)$, defined as the size of the largest collection $\A \subset 2^{[n]}$ not containing any poset $P \in \p$ as a subposet. Analogously, $\La^{\#}(n,\p)$ denotes the size of the largest family $\A \subset 2^{[n]}$ not containing any poset $P \in \p$ as an induced subposet.  In the case that $\p$ contains just one poset, $P$, or two posets, $P$ and $Q$, we simply write $\La(n,P)$ or $\La(n,P,Q)$ instead of $\La(n,\p)$.  Problems of this type are motivated by the famous theorem of Sperner \cite{sperner1928satz}:

\begin{theorem}
Let $\A \subset 2^{[n]}$ be an antichain, then 
\begin{displaymath}
\abs{\A} \le \binom{n}{\floor{n/2}}.
\end{displaymath}
Moreover, equality is attained if and only if $\A$ consists of a full level of size $\floor{n/2}$ or $\ceil{n/2}$.
\end{theorem}

Erd\H{o}s \cite{erdos1945lemma} extended this theorem by showing that if a family, $\A$, contains no chain of length $k+1$, then $\abs{\A} \le \Sigma(n,k)$.  If we define the $k+1$-path poset $P_{k+1}$ by $k+1$ elements $x_1,x_2,\dots,x_{k+1}$ with the relations $x_1 \le x_2 \le \dots \le x_{k+1}$,  then in our language, Erd\H{o}s's result states that $\La(n,P_{k+1}) = \Sigma(n,k)$.   For a variety of posets, $P$, the value of $\La(n,P)$ has been determined asymptotically. These include the $r$-fork $V_r$ \cite{de2007largest}, all crowns $O_{2k}$ except for $k \in \{3,5\}$ \cite{lu2014crown}, the $N$ poset \cite{griggs2008no} and many others.  An important asymptotic result is a theorem of Bukh \cite{bukh2009set} which yields the asymptotic bound for any tree, $\La(n,T)= (h(T)-1)\binom{n}{\floor{n/2}}(1+ O(1/n))$ where $h(T)$ is the height of the tree. This result was extended to induced trees by Boehnlein and Jiang \cite{boehnlein2012set}.  A central open problem is to determine the asymptotic value of the diamond poset (see, for example \cite{axenovich2012q,kramer2013diamond,manske2013three}).   

Fewer exact results are known.   Already, in their paper introducing the $\La$ function, Katona and Tarjan \cite{katona1983extremal} proved that $\La(n,V,\Lambda) = \La^{\#}(n,V,\Lambda) =  2 \binom{n-1}{\floor{\frac{n-1}{2}}}$, where $V$ and $\Lambda$ are the $2$-fork and $2$-brush, respectively.  Define the butterfly poset, $B$, by $4$ elements $a,b,c,d$ with $a,b \le c,d$. Of central importance to the present paper is a theorem of De Bonis, Katona and Swanepoel \cite{de2005largest} showing that $\La(n,B) = \Sigma(n,2)$.  More recently, several other exact results have been obtained. Burcsi and Nagy \cite{burcsi2013method} obtained the exact bound for multiple posets and introduced a method of creating further posets whose $\La$ function could be calculated exactly. Griggs, Li and Lu \cite{griggs2012diamond}  determined exact results for the $k$-diamond, $D_k$ ($w \le x_1,x_2,\dots,x_k \le z$), for an infinite set of values of $k$.  They also obtained exact results for harp posets, $H(l_1,l_2,\dots,l_k)$, defined by $k$ chains of length $l_i$ between two fixed elements, in the case when the $l_i$ are all distinct.  

Our first new result is a strengthening of the theorem of De Bonis, Katona and Swanepoel on the butterfly poset.  Namely,  we introduce a poset $S$ which contains the butterfly as a strict subposet and prove that, nonetheless, the same bound holds.  This poset, which we call the ``skew''-butterfly, is defined by $5$ elements, $a,b,c,d,e,$ with $a,b \le c,d$ and $b \le e \le d$ (see Figure \ref{skewpic}). 

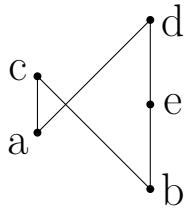
\begin{figure}[h]
\label{skewpic}
\begin{center}
\begin{tikzpicture}[line cap=round,line join=round,>=triangle 45,x=1.0cm,y=1.0cm, scale = 0.75]

\draw (1,1)-- (1,2);
\draw (3,1.5)-- (3,3);
\draw (1,1)-- (3,3);
\draw (1,2)-- (3,0);
\draw (3,1.5)-- (3,0);
\begin{scriptsize}
\fill [color=black] (1,1) circle (2.0pt);
\draw[color=black] (0.65,0.80) node {$\mathbf{\mbox{\Large a}}$};
\fill [color=black] (1,2) circle (2.0pt);
\draw[color=black] (0.65,2.1) node {$\mathbf{\mbox{\Large c}}$};
\fill [color=black] (3,1.5) circle (2.0pt);
\draw[color=black] (3.4,1.5) node {$\mathbf{\mbox{\Large e}}$};
\fill [color=black] (3,3) circle (2.0pt);
\draw[color=black] (3.4,3) node {$\mathbf{\mbox{\Large d}}$};
\fill [color=black] (3,0) circle (2.0pt);
\draw[color=black] (3.4,0) node {$\mathbf{\mbox{\Large b}}$};
\end{scriptsize}
\end{tikzpicture}
\end{center}
\caption{The skew-butterfly poset}
\label{figure 0}
\end{figure}

\begin{theorem}
\label{skew}
Let $n \ge 3$, then we have
\begin{displaymath}
\La(n,S) = \Sigma(n,2).
\end{displaymath}
\end{theorem}
A construction matching this bound is given by taking two consecutive middle levels of $2^{[n]}$.   With this result (and all of the others) we also get the corresponding LYM-type inequality if we assume $\varnothing$ and $[n]$ are not in the family.

\begin{theorem}
\label{skewbutterfly}
Let $n \ge 3$ and $\A \subset 2^{[n]}$  be a collection of sets not containing $S$ as a subposet, and assume that $\varnothing,[n] \notin \A$, then
\begin{displaymath}
\sum_{A\in \A} \frac{1}{\binom{n}{\abs{A}}} \le 2.
\end{displaymath}
\end{theorem}
For the proof of Theorem \ref{skewbutterfly}, we consider the set of intervals along a cyclic permutation (following Katona \cite{katona1972simple}).  We partition these intervals into chains and consider the interactions of consecutive chains in the partition.   The method and the proof of this result are given in Section $2$.

We now mention some notable properties of $S$.  It is one of the two posets whose Hasse diagram is a $5$-cycle.  The other is the harp, $H(4,3)$, and $\La(n,H(4,3))$ was determined exactly in the paper of Griggs, Li and Lu \cite{griggs2012diamond}(the $4$-cycles are $B$ and $D_2$).  The skew-butterfly is contained in the $X$ ($a,b \le c \le d,e$), a tree of height $3$, like $B$, and so its asymptotics are determined by Bukh's theorem.  The exact value of $\La(n,S)$ cannot be determined by the double chain method of Burcsi and Nagy \cite{burcsi2013method} because one can find $5$ sets on a double chain with no copy of $S$.   Finally, if we subdivide any of the edges $ac, ad$ or $bc$ in the Hasse diagram of $S$, we get a poset for which there is a construction of size larger than $\Sigma(n,2)$.

Next, we consider a generalization of De Bonis, Katona and Swanepoel's theorem in a different direction.  If instead of forbidding $B$, we forbid the pair of posets $Y$ and $Y'$ where $Y$ is the poset on $4$ elements $w,x,y,z$ with $w \le x \le y,z$ and $Y'$ is the same poset but with all relations reversed, then $\La(n,Y,Y') = \La(n,B) = \Sigma(n,2)$.  This result is already implicit in the proof of De Bonis, Katona and Swanepoel.  We extend the result by considering the posets $Y_k$ and $Y_k'$ defined by $k+2$ elements $x_1,x_2,\dots,x_{k},y,z$ with $x_1 \le x_2 \le \dots \le x_k \le y,z$ and its reverse (so $Y=Y_2$ and $V=Y_1$).  We prove

\begin{theorem}
\label{Yk}
Let $k \ge 2$ and $n \ge k+1$, then
\begin{displaymath}
\La(n,Y_k,Y_k') = \Sigma(n,k).
\end{displaymath}
\end{theorem}
A construction matching this bound is given by taking $k$ consecutive middle levels of $2^{[n]}$. We also have the LYM-type inequality:

\begin{theorem}
\label{YkLYM}
Let $k \ge 2$ and $n \ge k+1$.  Assume that $\A \subset 2^{[n]}$ contains neither $Y_k$ nor $Y_k'$ as a subposet, and $\varnothing,[n] \notin \A$, then
\begin{displaymath}
\sum_{A\in \A} \frac{1}{\binom{n}{\abs{A}}} \le k.
\end{displaymath}
\end{theorem}

We note that, again, the double chain method does not work for these pairs because one can have $2k+1$ sets on a double chain with no $Y_k$ and no $Y_k'$ by taking them consecutively on the secondary chain.  We also note that, for this particular result, we can find another proof using the chain partitioning method of Griggs, Li and Lu \cite{griggs2012diamond} in addition to the approach described in this paper.  

Finally, we consider the more difficult induced case.  We prove

\begin{theorem}
\label{induced}
For $n \ge 3$, we have
\begin{displaymath}
\La^{\#}(n,Y,Y') = \Sigma(2,n). 
\end{displaymath}
\end{theorem}

We also have the LYM-type inequality:

\begin{theorem}
\label{inducedLYM}
Assume that $\A \subset 2^{[n]}$ contains neither $Y$ nor $Y'$ as an induced subposet, and $\varnothing,[n] \notin \A$, then
\begin{displaymath}
\sum_{A\in \A} \frac{1}{\binom{n}{\abs{A}}} \le 2.
\end{displaymath}
\end{theorem}

To prove Theorem \ref{inducedLYM}, we introduce a second chain partitioning argument along the cycle.  These partitions may be thought of as the analogue of orthogonal symmetric chain partitions \cite{shearer1979probabilities} for the cycle. The method and the proof of Theorem  \ref{induced} are given in Section $3$.  Observe that because $B$ is a subposet of both $Y$ and $Y'$, and the inequality $\La(n,\p) \le \La^{\#}(n,\p)$ always holds, we again have a direct generalization of the result of De Bonis, Katona and Swanepoel.  Note that the problem of determining $\La^{\#}(n,B)$ is fundamentally different.  In this case, one can construct a larger family by adjoining distance $4$ constant weight codes \cite{graham1980lower} to two consecutive full levels.  Here an exact result is difficult to determine because the corresponding coding theory question is a long standing open problem.   Some asymptotic results for this problem may be found in \cite{patkos2014induced}.

The paper is organized as follows.  In the second section we introduce the first chain decomposition and determine $\La(n,S)$.  In the third section we use the same decomposition to find $\La(n,Y_k,Y_k')$ for all $k \ge 2$.  In the last section we introduce the second decomposition and show that $\La^\#(n,Y,Y') = \Sigma(n,2)$.

\section{Forbidding $S$ and the first cycle decomposition}
\label{S}
A cyclic permutation, $\sigma$, is a cyclic ordering $x_1,x_2,\dots, x_n,x_1$ of the elements of $[n]$.   We refer to the sets $\{x_i,x_{i+1},\dots,x_{i+t}\}$, with addition taken modulo $n$, as intervals along the cyclic permutation.  For our purpose we will not consider $\varnothing$ or $[n]$ to be intervals.    The following lemma is the essential ingredient of the proof of Theorem \ref{skewbutterfly}:

\begin{lemma}
\label{lemma:mainlemma}
If $\A$ is  a collection of intervals along a cyclic permutation $\sigma$ of $[n]$ which does not contain $S$ as a subposet, then 
\begin{displaymath}
\abs{\A} \le 2n.
\end{displaymath} 
\end{lemma}

To prove Lemma \ref{lemma:mainlemma} we will work with a decomposition of the intervals along $\sigma$ into maximal chains.   Set $\C_i = \{ \{x_i\}, \{x_i,x_{i-1}\}, \{x_i,x_{i-1},x_{i+1}\},\dots, \{x_i,x_{i-1},\dots,x_{i+n/2-1}\} \}$ when $n$ is even, and set $\C_i = \{\{x_i\}, \{x_i,x_{i-1}\}, \{x_i,x_{i-1},x_{i+1}\},\dots, \{x_i,x_{i-1},\dots,x_{i-(n-1)/2}\}\}$ when $n$ is odd, where $1 \le i \le n$ (See Figure \ref{figure 1}).  Observe that the set of chains $\{\C_i\}_{i=1}^n$ forms a partition of the intervals along $\sigma$.  We will refer to this partition as the \emph {chain decomposition} of $\sigma$. Additionally, chains corresponding to consecutive elements of $\sigma$ are called consecutive chains.  

\begin{figure}[h]
\begin{center}
\begin{tikzpicture}[line cap=round,line join=round,>=triangle 45, x=1.0cm,y=1.0cm, scale = 0.75]
\draw (2,3)-- (3,2);
\draw [line width=3.6pt] (4,3)-- (3,2);
\draw (4,3)-- (5,2);
\draw [line width=3.6pt] (5,2)-- (6,3);
\draw (6,3)-- (7,2);
\draw [line width=3.6pt] (7,2)-- (8,3);
\draw (8,3)-- (9,2);
\draw (2,3)-- (3,4);
\draw [line width=3.6pt] (4,3)-- (3,4);
\draw (5,4)-- (4,3);
\draw [line width=3.6pt] (5,4)-- (6,3);
\draw (6,3)-- (7,4);
\draw [line width=3.6pt] (7,4)-- (8,3);
\draw (8,3)-- (9,4);
\draw [line width=3.6pt] (9,2)-- (10,3);
\draw [line width=3.6pt] (3,4)-- (4,5);
\draw (5,4)-- (4,5);
\draw [line width=3.6pt] (5,4)-- (6,5);
\draw [line width=3.6pt] (7,4)-- (8,5);
\draw (9,4)-- (8,5);
\draw (10,3)-- (11,4);
\draw [line width=3.6pt] (9,4)-- (10,5);
\draw (10,5)-- (11,4);
\draw [line width=3.6pt] (9,4)-- (10,3);
\draw [line width=3.6pt] (11,4)-- (12,5);
\draw (1,2)-- (0,1);
\draw [line width=3.6pt] (1,2)-- (2,1);
\draw (3,2)-- (2,1);
\draw [line width=3.6pt] (3,2)-- (4,1);
\draw (5,2)-- (4,1);
\draw (7,2)-- (6,1);
\draw [line width=3.6pt] (7,2)-- (8,1);
\draw (9,2)-- (8,1);
\draw [line width=3.6pt] (9,2)-- (10,1);
\draw (11,2)-- (10,1);
\draw [line width=3.6pt] (11,2)-- (12,1);
\draw (13,2)-- (12,1);
\draw [line width=3.6pt] (13,2)-- (14,1);
\draw (10,3)-- (11,2);
\draw [line width=3.6pt] (11,4)-- (12,3);
\draw (13,4)-- (12,3);
\draw [line width=3.6pt] (12,3)-- (11,2);
\draw (13,2)-- (12,3);
\draw [line width=3.6pt] (13,4)-- (14,3);
\draw [line width=3.6pt] (14,3)-- (13,2);
\draw (14,1)-- (15,2);
\draw (15,2)-- (14,3);
\draw (12,5)-- (13,4);
\draw [line width=3.6pt] (15,2)-- (16,3);
\draw (16,3)-- (17,4);
\draw (14,3)-- (15,4);
\draw [line width=3.6pt] (15,4)-- (16,5);
\draw (16,5)-- (17,4);
\draw [line width=3.6pt] (15,4)-- (16,3);
\draw (15,4)-- (14,5);
\draw [line width=3.6pt] (14,5)-- (13,4);
\draw (14,5)-- (15,6);
\draw [line width=3.6pt] (15,6)-- (16,5);
\draw [line width=3.6pt] (10,5)-- (9,6);
\draw [line width=3.6pt] (9,6)-- (10,7);
\draw [line width=3.6pt] (12,5)-- (11,6);
\draw [line width=3.6pt] (12,7)-- (11,6);
\draw (6,5)-- (7,6);
\draw [line width=3.6pt] (7,6)-- (8,7);
\draw (8,5)-- (9,6);
\draw [line width=3.6pt] (7,6)-- (8,5);
\draw (4,5)-- (5,6);
\draw [line width=3.6pt] (5,6)-- (6,7);
\draw [line width=3.6pt] (5,6)-- (6,5);
\draw (7,6)-- (6,7);
\draw (9,6)-- (8,7);
\draw (11,6)-- (10,7);
\draw (11,6)-- (10,5);
\draw (12,7)-- (13,6);
\draw (13,6)-- (12,5);
\draw [line width=3.6pt] (13,6)-- (14,5);
\draw (15,6)-- (14,7);
\draw [line width=3.6pt] (13,6)-- (14,7);
\draw [line width=3.6pt] (5,2)-- (6,1);
\draw [line width=3.6pt] (2,3)-- (1,2);
\draw [line width=3.6pt] (15,6)-- (16,7);
\draw (16,7)-- (17,6);
\draw (17,6)-- (16,5);
\draw [line width=3.6pt] (17,4)-- (18,5);
\draw (18,5)-- (19,6);
\draw [line width=3.6pt] (17,6)-- (18,5);
\draw (19,6)-- (18,7);
\draw [line width=3.6pt] (18,7)-- (17,6);
\draw [line width=3.6pt] (19,6)-- (20,7);
\draw [line width=2.4pt,dash pattern=on 3pt off 3pt] (4.1,5.1)-- (19,6);
\draw [line width=2.4pt,dash pattern=on 3pt off 3pt] (17,4)-- (2,3);
\draw [line width=2.4pt,dash pattern=on 3pt off 3pt] (15,2)-- (0,1);
\draw (12,8)-- (6,7);
\draw (12,8)-- (8,7);
\draw (12,8)-- (10,7);
\draw (12,8)-- (12,7);
\draw (12,8)-- (14,7);
\draw (12,8)-- (16,7);
\draw (12,8)-- (18,7);
\draw (12,8)-- (19.98,6.98);
\draw (7,0)-- (8,1);
\draw (7,0)-- (2,1);
\draw (7,0)-- (0,1);
\draw (7,0)-- (10,1);
\draw (7,0)-- (12,1);
\draw (7,0)-- (14,1);
\draw (7,0)-- (4,1);
\draw (7,0)-- (6,1);
\end{tikzpicture}
\end{center}
\caption{The chain decomposition is marked with bold lines on the poset of intervals along $\sigma$. The dashed lines indicate how the chains wrap around.}
\label{figure 1}
\end{figure}
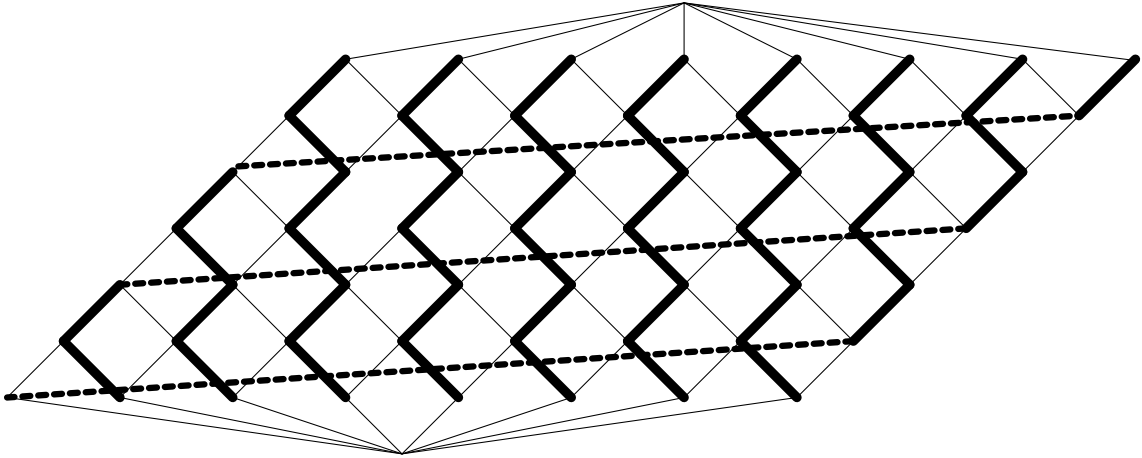

If $\A$ does not contain $S$ as a subposet, and $\C$ is a chain from the chain decomposition of $\sigma$, then it is easy to see that $\abs{\A\cap \C} \le 4$.   We will classify the chains in the chain decomposition by their intersection pattern with $\A$.   If $\abs{\A \cap \C}=k$, then we say $\C$ is of type $k$.    When $k=3$ we distinguish $3$ cases (See Figure \ref{figure 2} for an example of each case). If $\C$ contains exactly $3$ elements of $\A$, not all occurring consecutively on $\C$, then we say $\C$ is type $3^S$ ($S$ for separated).  If $\C$ has exactly 3 elements of $\A$ occurring consecutively with two sets of odd size, then $\C$ is type $3^R$ (facing right).  If $\C$ has exactly 3 elements of $\A$ occurring consecutively with two sets of even size, then $\C$ is type $3^L$ (facing left).

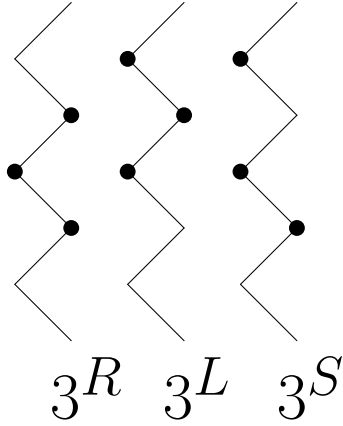
\begin{figure}[h]
\begin{center}
\begin{tikzpicture}[line cap=round,line join=round,>=triangle 45,x=1.0cm,y=1.0cm, scale = 0.75]
\draw (2,1)-- (1,2);
\draw (2,3)-- (1,2);
\draw (2,3)-- (1,4);
\draw (1,4)-- (2,5);
\draw (2,5)-- (1,6);
\draw (1.43,0.90
) node[anchor=north west] {$\mathbf{\mbox{\huge $3^R$}}$};
\draw (1,6)-- (2,7);
\draw (4,7)-- (3,6);
\draw (4,5)-- (3,4);
\draw (3,6)-- (4,5);
\draw (3,4)-- (4,3);
\draw (4,3)-- (3,2);
\draw (3,2)-- (4,1);
\draw (6,7)-- (5,6);
\draw (5,6)-- (6,5);
\draw (5,4)-- (6,5);
\draw (5,4)-- (6,3);
\draw (6,3)-- (5,2);
\draw (5,2)-- (6,1);
\draw (3.43,0.90) node[anchor=north west] {$\mathbf{\mbox{\huge $3^L$}}$};
\draw (5.45,0.90) node[anchor=north west] {$\mathbf{\mbox{\huge $3^S$}}$};
\begin{scriptsize}
\fill [color=black] (2,3) circle (4pt);
\fill [color=black] (1,4) circle (4pt);
\fill [color=black] (2,5) circle (4pt);
\fill [color=black] (3,6) circle (4pt);
\fill [color=black] (4,5) circle (4pt);
\fill [color=black] (3,4) circle (4pt);
\fill [color=black] (5,6) circle (4pt);
\fill [color=black] (5,4) circle (4pt);
\fill [color=black] (6,3) circle (4pt);
\end{scriptsize}
\end{tikzpicture}
\end{center}
\caption{An example of chains of types $3^R$, $3^L$ and $3^S$ are drawn. The elements of $\A \cap \C$ are highlighted for each type.}
\label{figure 2}
\end{figure}

We will now prove a sequence of lemmas showing which types of chains can occur consecutively in the chain decomposition of $\sigma$.  These lemmas will let us disregard the exact intersection pattern of $\A$ with the chains and allow us to work instead with the sequence of chain types.

\begin{lemma}
\label{lemma:technical1}
Let $\C_i$ and $\C_{i+1}$ be two consecutive chains in the chain decomposition of a cyclic permutation.  If $\C_i$ is of type $4, 3^R$ or $3^S$, then $\abs{\A \cap \C_{i+1}} \le 1$.
\end{lemma}
\begin{proof}
First, note that if $\C_i$ is of type 4, then we can remove a set from $\A \cap \C_i$ to make it type $3^S$.  Hence, we may assume that $\C_i$ is of type $3^S$ or $3^R$. 

In order to reduce case analysis, we will now argue that we only need to consider certain configurations of sets from $\A$ in $\C_i \cup \C_{i+1}$.  Consider the Hasse diagram  of $\C_i \cup \C_{i+1}$ as a graph (See Figure \ref{figure 3}). Call the vertices corresponding to sets in $\A$ occupied and the rest unoccupied.  If either the top or bottom vertex in the chain is occupied, then we extend $\C_i \cup \C_{i+1}$ in both directions maintaining the same relations between adjacent levels. Then, every occupied vertex either has degree $2$ or degree $4$.  We will see that it is sufficient to consider the case when only degree $2$ vertices are occupied.   Indeed, if instead of taking a degree $4$ vertex, we take an adjacent unoccupied degree $2$ vertex, then no additional containments are introduced.   If $\C_i$ is of type $3^R$ or $3^S$, then every occupied vertex of degree $4$ can be replaced by a distinct adjacent unoccupied vertex of degree 2 (This cannot be done if $\C_i$ is type $3^L$).  Thus, we may assume that all of the occupied vertices in $\C_i$ from the the Hasse diagram of $\C_i \cup \C_{i+1}$ have degree $2$.

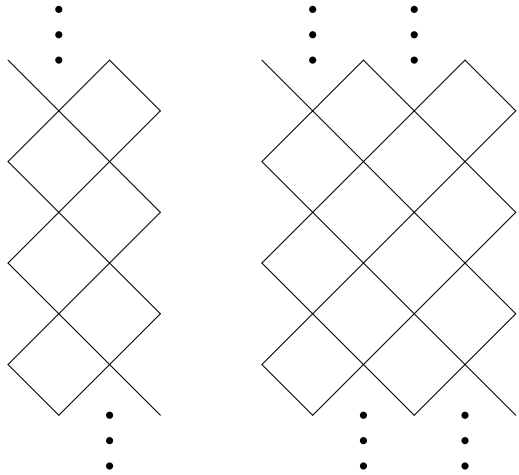
\begin{figure}[h]
\begin{center}
\begin{tikzpicture}[line cap=round,line join=round,>=triangle 45,x=1.0cm,y=1.0cm, scale = 1.35]
\draw (1,0.5)-- (0.5,1);
\draw (1,1.5)-- (0.5,1);
\draw (1,1.5)-- (1.5,1);
\draw (1.5,1)-- (1,0.5);
\draw (1.5,1)-- (2,0.5);
\draw (1.5,1)-- (2,1.5);
\draw (1,1.5)-- (0.5,2);
\draw (0.5,2)-- (1,2.5);
\draw (1,2.5)-- (1.5,2);
\draw (1.5,2)-- (1,1.5);
\draw (1.5,2)-- (2,1.5);
\draw (1.5,2)-- (2,2.5);
\draw (2,2.5)-- (1.5,3);
\draw (1.5,3)-- (1,2.5);
\draw (1,2.5)-- (0.5,3);
\draw (0.5,3)-- (1,3.5);
\draw (1,3.5)-- (1.5,3);
\draw (1.5,3)-- (2,3.5);
\draw (1,3.5)-- (0.5,4);
\draw (1.5,4)-- (1,3.5);
\draw (1.5,4)-- (2,3.5);
\draw (3.5,0.5)-- (3,1);
\draw (3.5,1.5)-- (3,1);
\draw (3.5,1.5)-- (4,1);
\draw (4,1)-- (3.5,0.5);
\draw (4,1)-- (4.5,0.5);
\draw (4,1)-- (4.5,1.5);
\draw (3.5,1.5)-- (3,2);
\draw (3,2)-- (3.5,2.5);
\draw (3.5,2.5)-- (4,2);
\draw (4,2)-- (3.5,1.5);
\draw (4,2)-- (4.5,1.5);
\draw (4,2)-- (4.5,2.5);
\draw (4.5,2.5)-- (4,3);
\draw (4,3)-- (3.5,2.5);
\draw (3.5,2.5)-- (3,3);
\draw (3,3)-- (3.5,3.5);
\draw (3.5,3.5)-- (4,3);
\draw (4,3)-- (4.5,3.5);
\draw (3.5,3.5)-- (3,4);
\draw (4,4)-- (3.5,3.5);
\draw (4,4)-- (4.5,3.5);
\draw (4.5,0.5)-- (5,1);
\draw (5,1)-- (4.5,1.5);
\draw (4.5,1.5)-- (5,2);
\draw (5,2)-- (4.5,2.5);
\draw (4.5,2.5)-- (5,3);
\draw (5,3)-- (4.5,3.5);
\draw (4.5,3.5)-- (5,4);
\draw (5,4)-- (5.5,3.5);
\draw (5.5,3.5)-- (5,3);
\draw (5,3)-- (5.5,2.5);
\draw (5.5,2.5)-- (5,2);
\draw (5,2)-- (5.5,1.5);
\draw (5.5,1.5)-- (5,1);
\draw (5,1)-- (5.5,0.5);
\begin{scriptsize}
\fill [color=black] (1.5,0) circle (1.0pt);
\fill [color=black] (1.5,0.25) circle (1.0pt);
\fill [color=black] (1.5,0.5) circle (1.0pt);
\fill [color=black] (1,4) circle (1.0pt);
\fill [color=black] (1,4.25) circle (1.0pt);
\fill [color=black] (1,4.5) circle (1.0pt);
\fill [color=black] (3.5,4) circle (1.0pt);
\fill [color=black] (3.5,4.25) circle (1.0pt);
\fill [color=black] (3.5,4.5) circle (1.0pt);
\fill [color=black] (4.5,4) circle (1.0pt);
\fill [color=black] (4.5,4.25) circle (1.0pt);
\fill [color=black] (4.5,4.5) circle (1.0pt);
\fill [color=black] (4,0) circle (1.0pt);
\fill [color=black] (4,0.25) circle (1.0pt);
\fill [color=black] (4,0.5) circle (1.0pt);
\fill [color=black] (5,0) circle (1.0pt);
\fill [color=black] (5,0.25) circle (1.0pt);
\fill [color=black] (5,0.5) circle (1.0pt);
\end{scriptsize}
\end{tikzpicture}
\end{center}
\caption{Hasse diagrams of $\C_i \cup \C_{i+1}$ and $\C_i \cup \C_{i+1} \cup \C_{i+2}$ are drawn.}
\label{figure 3}
\end{figure}

Let the sets in $\A \cap \C_i$ be $L,M$ and $N$ with $L \subset M \subset N$.  Assume, by contradiction, that there are two sets $A,B \in \A \cap \C_{i+1}$ with $A \subset B$.  We may assume that $A$ and $B$ correspond to degree $2$ vertices in $\C_i \cap \C_{i+1}$.   We will distinguish three cases by comparing the sizes of $A$ and $B$ with the size of $M$.   If $\abs{A}< \abs{M}<\abs{B}$, then  $L,M,N,A,B$ forms a skew-butterfly with $L,A \subset N,B$ and $L \subset M \subset N$. If $\abs{M}<\abs{A}<\abs{B}$, then $L,M,N,A,B$ forms a skew-butterfly with $L,M \subset N,B$ and $L \subset A \subset B$.  The case $\abs{A}<\abs{B}<\abs{M}$ is symmetric. It follows that there can be at most one set in $\A \cap \C_{i+1}$.

\end{proof}
\begin{lemma}
\label{lemma:technical2}
Let $\C_i,\C_{i+1}$ and $\C_{i+2}$ be three consecutive chains in the chain decomposition of a cyclic permutation.  If $\C_i$ is of type $4,3^R$ or $3^S$ and $\abs{\A \cap \C_{i+1}}=1$, then $\C_{i+2}$ is of type $0,1,2$ or $3^R$.
\end{lemma}

\begin{proof}
By contradiction, suppose $\C_i$ is type $4, 3^R$ or $3^S$, $\abs{\A \cap \C_{i+1}}=1$ and $\C_{i+2}$ is type $3^L, 3^S$ or 4. If $\C_i$ or $\C_{i+2}$ is of type 4, then we may disregard one set to make it type $3^S$.   By similar reasoning as used in Lemma \ref{lemma:technical1}, we may assume all occupied vertices on the Hasse diagram of $\C_i \cup \C_{i+1} \cup \C_{i+2}$ from $\C_i$ and $\C_{i+2}$ have degree 2.  Let $L,M,N$ be the three sets in $\A \cap \C_i$ in increasing order, and let $A,B,C$ be the three sets in $\A \cap \C_{i+2}$ in increasing order.  Without loss of generality, we may assume $\abs{M}>\abs{B}$. This, in turn, implies that $\abs{M}=\abs{B}+1$ for otherwise $L,M,N,A,B$ would be a skew-butterfly with $L,A \subset M,N$ and $A \subset B \subset N$.  We will consider the possible locations of the set $S \in \A\cap \C_{i+1}$ on $\C_{i+1}$.  If $\abs{S} \le \abs{B}$, then $N,S,A,B,C$ is a skew-butterfly with $A,S \subset N,C$ and $A \subset B \subset C$.  If $\abs{S}>\abs{B}$, then $L,M,N,S,A$ is a skew-butterfly with $L,A \subset N,S$ and $L \subset M \subset N$. Thus, in either case we have a contradiction.
\end{proof}

By symmetry, we also have the following corollaries of Lemmas \ref{lemma:technical1} and \ref{lemma:technical2}:

\begin{corollary}
\label{cor:technical1}
Let $\C_i$ and $\C_{i+1}$ be two consecutive chains in the chain decomposition of a cyclic permutation.  If $\C_{i+1}$ is of type $4, 3^L$ or $3^S$, then $\abs{\A \cap \C_i} \le 1$.
\end{corollary}

\begin{corollary}
\label{cor:technical2}
Let $\C_i,\C_{i+1}$ and $\C_{i+2}$ be three consecutive chains in the chain decomposition of a cyclic permutation.  If $\C_{i+2}$ is of type $4, 3^L$ or $3^S$ and $\abs{\A \cap \C_{i+1}}=1$, then $\C_i$ is of type $0,1,2$ or $3^L$.
\end{corollary}

We now have sufficient information about which consecutive chain types are allowed to  prove Lemma \ref{lemma:mainlemma}:

\begin{proof}[Proof of Lemma \ref{lemma:mainlemma}]
We must show that the average intersection of $\A$ with chains from the decomposition is at most $2$.  To this end, we will form groups of chains such that the number of sets from $\A$ in each group is at most twice the size of that group.  

First, consider chains of type 4.  If there is a sequence of chains alternating between type 4 and type 0 spanning every chain in the chain decomposition, then it is easy to see that the average is at most 2.  Otherwise, take each maximal group of consecutive chains alternating between type $0$ and type $4$, beginning and ending with a type $4$ chain. Call such a group a 4-0-4 pattern (it may just consist of a single chain of type $4$).  If the group has length $\ell$, then there are $2\ell+2$ sets contributed from $\A$.  We will add additional chains to this group to decrease the average to 2.  By Lemma \ref{lemma:technical1}, if the chain following the type 4 chain on either side is not type 0, then it must be type 1.  In this case, we add the type 1 chain to the group.  Otherwise, we have a type 0 chain followed by a chain of type 0,1,2 or 3.  If it is type 3, we add both the type 0 and type 3 chain to our group.  Otherwise, we just add the type 0 chain.  In any case, if we have added $k$ more chains to our group (on both sides of the 4-0-4 pattern), then we have added a total of at most $2k-2$ more sets from $\A$.  Thus, in total, the group now consists of $k+\ell$ chains having at most $2k + 2\ell$ sets from $\A$, as desired.

Now, consider any remaining type 3 chain.  Lemma \ref{lemma:technical1} and Corollary \ref{cor:technical1} ensure that it has a type 1 or type 0 chain on at least one side (right or left).  By Lemma \ref{lemma:technical2} and Corollary \ref{cor:technical2} and by the previous grouping of the chains of type $4$, we know that this chain was not used by any group consisting of chains of type $4$.  Thus, every type $3$ chain may be grouped with its adjacent type $1$ or $0$ chain.  All remaining chains in the decomposition have at most 2 sets from $\A$ and so we may group them all together.

\end{proof}

We now derive the LYM-type inequality, Theorem \ref{skewbutterfly}, from Lemma \ref{lemma:mainlemma}.  

\begin{proof}
We will double count pairs $(A,\sigma)$ where $A \in \A$ and $\sigma$ is a cyclic permutation of $[n]$.  Let $f(A,\sigma)$ be the indicator function for $A\in \A$ and $A$ being an interval along $\sigma$.  For each $A \in \A$, there are $\abs{A}!(n-\abs{A})!$ cyclic permutations containing $A$ as an interval.  It follows that
\begin{displaymath}
\sum_{A \in \A} \sum_\sigma f(A,\sigma) = \sum_{A \in \A} \abs{A}!(n-\abs{A})!.
\end{displaymath}
On the other hand, Lemma \ref{lemma:mainlemma} implies
\begin{displaymath}
\sum_\sigma \sum_{A \in \A} f(A,\sigma) \le \sum_\sigma 2n = 2n!. 
\end{displaymath}
Dividing through by $n!$ gives
\begin{displaymath}
\sum_{A \in \A} \frac{1}{\binom{n}{\abs{A}}} \le 2,
\end{displaymath}
as desired.
\end{proof}

Finally, we deduce Theorem \ref{skew} from Theorem \ref{skewbutterfly}.

\begin{proof}
If $\A$ contains neither $[n]$ nor $\varnothing$, then the result follows easily from Theorem \ref{skew}. If $\A$ contains $[n]$, but there is an $n-1$ element set $A$ not contained in $\A$, then replacing $[n]$ with $A$ in $\A$ introduces no new relations and so yields another family of the same size without a skew-butterfly. Thus, in this case, Theorem \ref{skew} again yields the result.  If $\A$ contains $[n]$ and the entire $n-1^{st}$ level, let $\A' = \{A \in \A: \abs{A} \le n-2\}$.  Then, $\A'$ is an antichain, for otherwise we would have a skew-butterfly. Thus, $\abs{\A'} \le \binom{n}{\floor{n/2}}$ by Sperner's Theorem and so $\abs{\A} \le \binom{n}{\floor{n/2}} + n + 1$.  For $n \ge 5$ this implies $\abs{\A} \le  \binom{n}{\floor{n/2}} +  \binom{n}{\floor{n/2}+1}$.  An analogous argument works for the case when $\varnothing \in \A$.   If $n=4$ we give another argument (We are still assuming $\A$ contains all $n-1$ element sets).   If $\A'$ is a full level, then $\A$ contains  a skew-butterfly.  If $\A'$ is not a full level, then the equality case of Sperner's theorem implies $\abs{\A'} \le \binom{n}{\floor{n/2}}-1$, and so $\abs{\A} \le n + \binom{n}{\floor{n/2}}$ which yields the required bound when $n=4$.   The case $n=3$ is easily checked by hand.  
\end{proof}

We end this section by mentioning the relation between this approach and the double chain method.  It is not hard to see that a double chain has the exact same poset structure as two consecutive chains in the chain decomposition described above.  Namely, the degree 2 vertices from the Hasse diagram of consecutive chains correspond to the sets from the secondary chain of a double chain.  It follows that any forbidden subposet result that can be determined exactly with the double chain method can also be determined exactly using a decomposition of a cyclic permutation, and, thus, chain decompositions of the cycle may be viewed as a generalization of the double chain method.   

\section{Forbidding $Y_k$ and $Y_k'$}
\label{YkYk'}
We will use the same decomposition of the cycle as in Section \ref{S}.  The new bound we must prove is 

\begin{lemma}
\label{lemma:mainlemma2}
If $\A$ is a collection of intervals along a cyclic permutation $\sigma$ of $[n]$ which does not contain $Y_k$ or $Y_k'$ as a subposet, then 
\begin{displaymath}
\abs{\A} \le kn.
\end{displaymath} 
\end{lemma}

As before, we will consider groups of consecutive chains.    Each chain, $\C$, with $k+1$ sets in $\C \cap \A$ is characterized by whether the second largest element in $\A \cap \C$ faces left or faces right (has even or odd cardinality, respectively).  We say that a chain with $k+1$ elements of $\A$ is of type $k+1^R$ if the second largest element faces right and $k+1^L$ if it faces left.  

\begin{lemma}
\label{lemmaX}
Let $\C_i$ and $\C_{i+1}$ be consecutive chains in the decomposition.  If $\C_i$ is of type $k+1^R$, then $\abs{\A \cap \C_{i+1}} \le k-1$, and  $\abs{\A \cap \C_{i+1}} = k-1$ implies that the largest element of $\A \cap \C_{i+1}$ is the same size as the second largest element of $\A \cap \C_{i}$.
\end{lemma}

\begin{proof}
Let $A$ be the second smallest set in $\A \cap \C_i$ and $B$ be the second largest. Let $Y$ be the set of size $\abs{B}$ in $\C_{i+1}$, and if $A$ is degree $2$ (left), then let $X$ be the set of size $\abs{A}-1$ in $\C_{i+1}$. If $A$ is degree $4$, then let $X$ be the set of size $\abs{A}$ in $\C_{i+1}$. In either case, let $\R$ be the collection of those sets in $\C_{i+1}$ (not necessarily in $\A$) having sizes strictly between $\abs{X}$ and $\abs{Y}$ (See Figure \ref{figure 5}).  Every set in $\C_{i+1} \cap \A$  must lie in $\R \cup \{X\}\cup\{Y\}$ for otherwise we would have a $Y_k$ or $Y_k'$.    Now, $\abs{\A \cap \R} \le k-2$ for otherwise we would have a $k+2$ chain (actually, $\abs{\A \cap \R} \le k-3$ in the case $A$ is degree $4$).  If we take $k-1$ sets from $\R \cup \{X\}$, then we have a $Y_k'$ and so  we can take at most $k-2$ sets total from $\R \cup \{X\}$.  It follows that $\abs{\A \cap \C_{i+1}} \le k-1$ with equality only if $Y \in \A$. 

\end{proof}

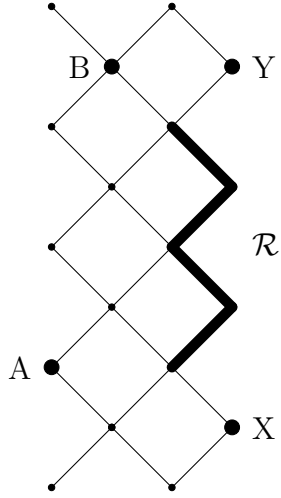
\begin{figure}[h]
\begin{center}
\begin{tikzpicture}[line cap=round,line join=round,>=triangle 45,x=1.0cm,y=1.0cm, scale = 0.8]
\draw (4.0,7.0)-- (3.0,8.0);
\draw (3.0,6.0)-- (4.0,7.0);
\draw (3.0,6.0)-- (4.0,5.0);
\draw (4.0,5.0)-- (3.0,4.0);
\draw (3.0,4.0)-- (4.0,3.0);
\draw (5.0,8.0)-- (6.0,7.0);
\draw (6.0,7.0)-- (5.0,6.0);  
\draw [line width=4.0pt] (5.0,6.0)-- (6.0,5.0);
\draw [line width=4.0pt] (6.0,5.0)-- (5.0,4.0);
\draw [line width=4.0pt] (5.0,4.0)-- (6.0,3.0);
\draw (5.0,8.0)-- (4.0,7.0);
\draw (4.0,7.0)-- (5.0,6.0);
\draw (5.0,6.0)-- (4.0,5.0);
\draw (4.0,5.0)-- (5.0,4.0);
\draw (5.0,4.0)-- (4.0,3.0);
\draw (4.0,3.0)-- (3.0,2.0);
\draw (3.0,2.0)-- (4.0,1.0);
\draw [line width=4.0pt] (6.0,3.0)-- (5.0,2.0);
\draw (5.0,2.0)-- (6.0,1.0);
\draw (4.0,3.0)-- (5.0,2.0);
\draw (5.0,2.0)-- (4.0,1.0);
\draw (3.1,7.35) node[anchor=north west] {B};
\draw (2.1,2.35) node[anchor=north west] {A};
\draw (4.0,1.0)-- (3.0,0.0);
\draw (6.0,1.0)-- (5.0,0.0);
\draw (4.0,1.0)-- (5.0,0.0);
\draw (6.15,7.35) node[anchor=north west] {Y};
\draw (6.15,1.35) node[anchor=north west] {X};
\draw (6.15,4.410097670924116) node[anchor=north west] {$\mathcal{R}$};
\begin{scriptsize}
\draw [fill=black] (3.0,8.0) circle (1.5pt);
\draw [fill=black] (4.0,7.0) circle (3.5pt);
\draw [fill=black] (3.0,6.0) circle (1.5pt);
\draw [fill=black] (4.0,5.0) circle (1.5pt);
\draw [fill=black] (3.0,4.0) circle (1.5pt);
\draw [fill=black] (4.0,3.0) circle (1.5pt);
\draw [fill=black] (5.0,8.0) circle (1.5pt);
\draw [fill=black] (6.0,7.0) circle (3.5pt);
\draw [fill=black] (5.0,6.0) circle (1.5pt);
\draw [fill=black] (6.0,5.0) circle (1.5pt);
\draw [fill=black] (5.0,4.0) circle (1.5pt);
\draw [fill=black] (6.0,3.0) circle (1.5pt);
\draw [fill=black] (3.0,2.0) circle (3.5pt);
\draw [fill=black] (4.0,1.0) circle (1.5pt);
\draw [fill=black] (5.0,2.0) circle (1.5pt);
\draw [fill=black] (6.0,1.0) circle (3.5pt);
\draw [fill=black] (3.0,0.0) circle (1.5pt);
\draw [fill=black] (5.0,0.0) circle (1.5pt);
\draw [fill=black] (4.0,1.0) circle (1.5pt);
\end{scriptsize}
\end{tikzpicture}
\end{center}
\caption{The sets $A$, $B$, $X$ and $Y$ are shown, and the collection $\mathcal{R}$ is marked in the case $A$ is degree $2$.}
\label{figure 5}
\end{figure}

By a symmetric argument we have
\begin{corollary}
\label{corY}
Let $\C_i$ and $\C_{i+1}$ be consecutive chains in the decomposition.  If $\C_{i+1}$ is of type $k+1^L$, then $\abs{\A \cap \C_{i}} \le k-1$, and  $\abs{\A \cap \C_{i}} = k-1$ implies that the largest element of $\A \cap \C_{i}$ is the same size as the second largest element of $\A \cap \C_{i+1}$.
\end{corollary}

\begin{lemma}
There are no $3$ consecutive chains $\C_i, \C_{i+1},\C_{i+2}$ such that $\C_i$ is type $k+1^R$, $\C_{i+1}$ is type $k-1$ and $\C_{i+1}$ is type $k+1^L$.
\end{lemma}

\begin{proof}
Since $\C_{i}$ is type $k+1^R$ and $\C_{i+2}$ is type $k+1^L$, the respective second largest elements of $\A \cap \C_i$ and $\A \cap \C_{i+2}$ must be of different sizes.  It follows from Lemma \ref{lemmaX} and Corollary \ref{corY} that we can have at most $k-2$ sets in $\A \cap \C_{i+1}$.
\end{proof}

We now have what we need to prove Lemma \ref{lemma:mainlemma2}.

\begin{proof}[Proof of Lemma \ref{lemma:mainlemma2}]
Every group of $3$ consecutive chains of type $k+1^R$, $\le k-2$ and $k+1^L$, respectively,  may be grouped together yielding a total of at most $3k$ sets on $3$ chains.  All remaining chains of type $k+1^R$ may be paired with a chain of at most $k-1$ sets from $\A$ following it, and all remaining chains of type $k+1^L$ may be paired with a chain of at most $k-1$ sets preceding it.  It follows that $\A$ consists of at most $kn$ intervals along the cyclic permutation $\sigma$.
\end{proof}

Theorem \ref{YkLYM} follows directly from Lemma \ref{lemma:mainlemma2} as before.  It remains to use Theorem \ref{YkLYM} to deduce Theorem \ref{Yk}.

\begin{proof}[Proof of Theorem \ref{Yk}]
Let $\A \subset 2^{[n]}$ be a $Y_k$ and $Y_k'$-free family.  If neither of $\varnothing$ and $[n]$ are in $\A$, then the result is immediate from Theorem \ref{YkLYM}.  If $\varnothing$ and $[n]$ are in $\A$, then $\A \setminus \{\varnothing,[n]\}$ is $k$-chain free and so has size at most $\Sigma(n,k-1)$ by Erd\H{o}s's theorem. Since $2 + \Sigma(n,k-1) \le \Sigma(n,k)$ for $n \ge k+1$ and $k\ge 2$, we are done.  Finally, suppose that $\varnothing \in \A$ and $[n] \not\in \A$.  If there is a singleton set $\{x\} \not\in \A$, then we may replace $\varnothing$ with $\{x\}$ and we are back in the first case.  Hence, we may assume that $\A$ contains every singleton set ($\binom{[n]}{1}\subset \A$).  Let $\A' = \A \setminus \{\{\varnothing\} \cup \binom{[n]}{1}\}$.  Now, $\A'$ is $k$-chain free, so again by Erd\H{o}s's theorem, $\abs{\A'} \le \Sigma(n,k-1)$.  It follows that $\abs{\A} \le 1 + n + \Sigma(n,k-1)$. If $\A'$ contains $k-1$ full levels, then we have a copy of $Y_k'$, so we may assume we do not.  However, then we may apply the equality case of Erd\H{o}s's theorem to obtain that $\abs{\A} \le n + \Sigma(n,k-1)$.  Finally, since $n \ge k+1$ implies that the $k^{th}$ largest level has size at least $n$, we have $\abs{\A} \le \Sigma(n,k)$, as desired.
\end{proof}

\section{Forbidding induced $Y$ and $Y'$ and second cycle decomposition}
As in the proof of Theorem \ref{skew}, we will need to prove a lemma which bounds the largest intersection of an induced $Y$,$Y'$-free family with the set of intervals along a cyclic permutation.

\begin{lemma}
\label{ortho}
If $\A$ is a collection of intervals along a cyclic permutation $\sigma$ of $[n]$ which does not contain $Y$ or $Y'$ as an induced subposet, then 
\begin{displaymath}
\abs{\A} \le 2n.
\end{displaymath} 
\end{lemma}

\begin{proof}
We will consider a different way of partitioning the chains along $\sigma$  from the one in the proofs of the previous theorems.  Let $\sigma$ be the ordering $x_1,x_2,\dots,x_n,x_1$.  Group the intervals along $\sigma$ into chains $\C_i = \{\{x_i\},\{x_i,x_{i+1}\}, \{x_i,x_{i+1},x_{i+2}\},\dots,\{x_i,x_{i+1},\dots,x_{i+n-1}\}\}$ where $1 \le i \le n$.  Observe that $\{\C_i\}_{i=1}^n$ is a partition of the intervals along $\sigma$.  

We now consider a second way of partitioning the intervals by setting \\  $\C_i' = \{\{x_i\},\{x_i,x_{i-1}\}, \{x_i,x_{i-1},x_{i-2}\},\dots, \{x_i,x_{i-1},\dots, x_{i-n+1}\}\}$ for $1 \le i \le n$.  Observe that $\{\C_i'\}_{i=1}^n$ is again a partition (See Figure \ref{Ortho_decomp}).

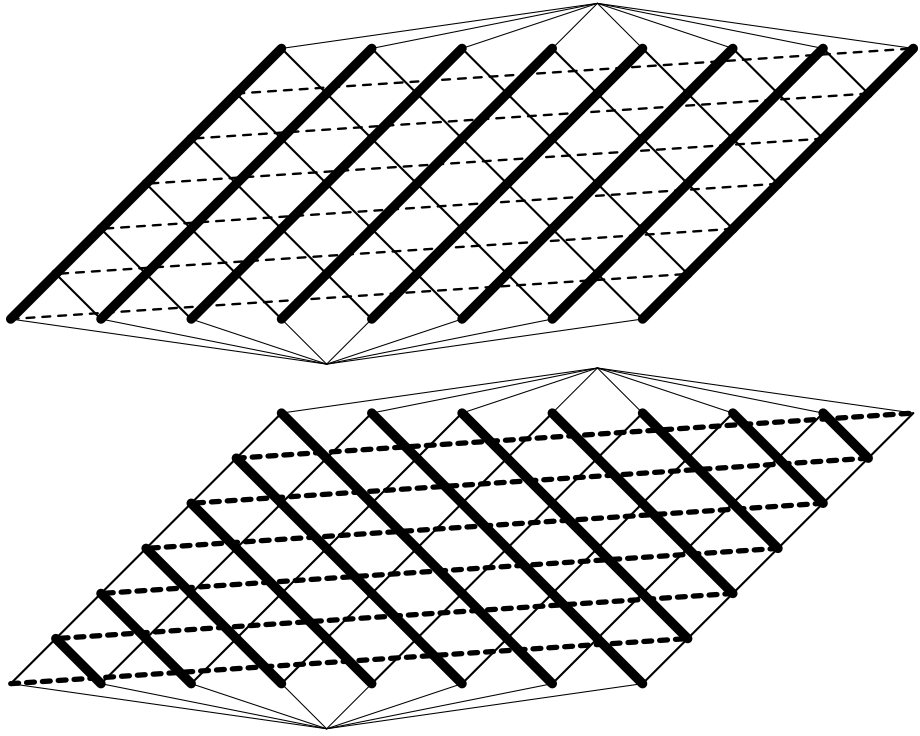
\begin{figure}[h]
\begin{center}
\begin{tikzpicture}[line cap=round,line join=round, x=1.0cm,y=1.0cm, scale = 0.60]

\draw [line width=3.6pt] (1.0,2.0)-- (2.0,3.0);
\draw [line width=0.8 pt] (2.0,3.0)-- (3.0,2.0);
\draw [line width=3.6pt] (4.0,3.0)-- (3.0,2.0);
\draw [line width=0.8 pt] (4.0,3.0)-- (5.0,2.0);
\draw [line width=3.6pt] (5.0,2.0)-- (6.0,3.0);
\draw [line width=0.8 pt] (6.0,3.0)-- (7.0,2.0);
\draw [line width=3.6pt] (7.0,2.0)-- (8.0,3.0);
\draw [line width=0.8 pt ] (8.0,3.0)-- (9.0,2.0);
\draw [line width=3.6pt] (2.0,3.0)-- (3.0,4.0);
\draw [line width=0.8 pt ] (4.0,3.0)-- (3.0,4.0);
\draw [line width=3.6pt] (5.0,4.0)-- (4.0,3.0);
\draw [line width=0.8 pt ] (5.0,4.0)-- (6.0,3.0);
\draw [line width=3.6pt] (6.0,3.0)-- (7.0,4.0);
\draw [line width=0.8 pt ] (7.0,4.0)-- (8.0,3.0);
\draw [line width=3.6pt] (8.0,3.0)-- (9.0,4.0);
\draw [line width=3.6pt] (9.0,2.0)-- (10.0,3.0);
\draw [line width=3.6pt] (3.0,4.0)-- (4.0,5.0);
\draw [line width=0.8 pt ] (5.0,4.0)-- (4.0,5.0);
\draw [line width=3.6pt] (5.0,4.0)-- (6.0,5.0);
\draw [line width=0.8 pt ] (7.0,4.0)-- (6.0,5.0);
\draw [line width=3.6pt] (7.0,4.0)-- (8.0,5.0);
\draw [line width=0.8 pt ] (9.0,4.0)-- (8.0,5.0);
\draw [line width=3.6pt] (10.0,3.0)-- (11.0,4.0);
\draw [line width=3.6pt] (9.0,4.0)-- (10.0,5.0);
\draw [line width=0.8 pt ] (10.0,5.0)-- (11.0,4.0);
\draw [line width=0.8 pt ] (9.0,4.0)-- (10.0,3.0);
\draw [line width=3.6pt] (11.0,4.0)-- (12.0,5.0);
\draw [line width=3.6pt] (4.0,5.0)-- (5.0,6.0);
\draw [line width=0.8 pt ] (6.0,5.0)-- (5.0,6.0);
\draw [line width=3.6pt] (6.0,5.0)-- (7.0,6.0);
\draw [line width=0.8 pt ] (8.0,5.0)-- (7.0,6.0);
\draw [line width=3.6pt] (8.0,5.0)-- (9.0,6.0);
\draw [line width=0.8 pt ] (10.0,5.0)-- (9.0,6.0);
\draw [line width=3.6pt] (10.0,5.0)-- (11.0,6.0);
\draw [line width=0.8 pt ] (11.0,6.0)-- (12.0,5.0);
\draw [line width=3.6pt] (12.0,5.0)-- (13.0,6.0);
\draw [line width=0.8 pt ] (11.0,4.0)-- (12.0,3.0);
\draw [line width=0.8 pt ] (12.0,5.0)-- (13.0,4.0);
\draw [line width=0.8 pt ] (13.0,6.0)-- (14.0,5.0);
\draw [line width=3.6pt] (11.0,2.0)-- (12.0,3.0);
\draw [line width=3.6pt] (12.0,3.0)-- (13.0,4.0);
\draw [line width=3.6pt] (13.0,4.0)-- (14.0,5.0);
\draw [line width=3.6pt] (14.0,5.0)-- (15.0,6.0);
\draw [line width=3.6pt] (13.0,2.0)-- (14.0,3.0);
\draw [line width=3.6pt] (15.0,4.0)-- (14.0,3.0);
\draw [line width=3.6pt] (15.0,4.0)-- (16.0,5.0);
\draw [line width=3.6pt] (16.0,5.0)-- (17.0,6.0);
\draw [line width=0.8 pt ] (15.0,6.0)-- (16.0,5.0);
\draw [line width=0.8 pt ] (14.0,5.0)-- (15.0,4.0);
\draw [line width=0.8 pt ] (13.0,4.0)-- (14.0,3.0);
\draw [line width=0.8 pt ] (12.0,3.0)-- (13.0,2.0);
\draw [line width=3.6pt] (15.0,2.0)-- (16.0,3.0);
\draw [line width=3.6pt] (17.0,4.0)-- (16.0,3.0);
\draw [line width=3.6pt] (17.0,4.0)-- (18.0,5.0);
\draw [line width=3.6pt] (18.0,5.0)-- (19.0,6.0);
\draw [line width=0.8 pt ] (14.0,3.0)-- (15.0,2.0);
\draw [line width=0.8 pt ] (15.0,4.0)-- (16.0,3.0);
\draw [line width=0.8 pt ] (16.0,5.0)-- (17.0,4.0);
\draw [line width=0.8 pt ] (17.0,6.0)-- (18.0,5.0);
\draw [line width=0.8 pt ] (10.0,3.0)-- (11.0,2.0);
\draw [line width=3.6pt] (5.0,6.0)-- (6.0,7.0);
\draw [line width=0.8 pt ] (7.0,6.0)-- (6.0,7.0);
\draw [line width=3.6pt] (7.0,6.0)-- (8.0,7.0);
\draw [line width=0.8 pt ] (9.0,6.0)-- (8.0,7.0);
\draw [line width=3.6pt] (9.0,6.0)-- (10.0,7.0);
\draw [line width=0.8 pt ] (11.0,6.0)-- (10.0,7.0);
\draw [line width=3.6pt] (11.0,6.0)-- (12.0,7.0);
\draw [line width=0.8 pt ] (12.0,7.0)-- (13.0,6.0);
\draw [line width=3.6pt] (13.0,6.0)-- (14.0,7.0);
\draw [line width=0.8 pt ] (14.0,7.0)-- (15.0,6.0);
\draw [line width=3.6pt] (15.0,6.0)-- (16.0,7.0);
\draw [line width=3.6pt] (17.0,6.0)-- (18.0,7.0);
\draw [line width=0.8 pt ] (16.0,7.0)-- (17.0,6.0);
\draw [line width=3.6pt] (19.0,6.0)-- (20.0,7.0);
\draw [line width=0.8 pt ] (18.0,7.0)-- (19.0,6.0);
\draw [line width=3.6pt] (6.0,7.0)-- (7.0,8.0);
\draw [line width=0.8 pt ] (8.0,7.0)-- (7.0,8.0);
\draw [line width=3.6pt] (8.0,7.0)-- (9.0,8.0);
\draw [line width=0.8 pt ] (10.0,7.0)-- (9.0,8.0);
\draw [line width=3.6pt] (10.0,7.0)-- (11.0,8.0);
\draw [line width=0.8 pt ] (12.0,7.0)-- (11.0,8.0);
\draw [line width=3.6pt] (12.0,7.0)-- (13.0,8.0);
\draw [line width=0.8 pt ] (13.0,8.0)-- (14.0,7.0);
\draw [line width=3.6pt] (14.0,7.0)-- (15.0,8.0);
\draw [line width=0.8 pt ] (15.0,8.0)-- (16.0,7.0);
\draw [line width=3.6pt] (16.0,7.0)-- (17.0,8.0);
\draw [line width=3.6pt] (18.0,7.0)-- (19.0,8.0);
\draw [line width=0.8 pt ] (17.0,8.0)-- (18.0,7.0);
\draw [line width=3.6pt] (20.0,7.0)-- (21.0,8.0);
\draw [line width=0.8 pt ] (19.0,8.0)-- (20.0,7.0);
\draw (14.0,9.0)-- (7.0,8.0);
\draw (14.0,9.0)-- (9.0,8.0);
\draw (14.0,9.0)-- (11.0,8.0);
\draw (14.0,9.0)-- (13.0,8.0);
\draw (14.0,9.0)-- (15.0,8.0);
\draw (14.0,9.0)-- (17.0,8.0);
\draw (14.0,9.0)-- (19.0,8.0);
\draw (14.0,9.0)-- (21.0,8.0);
\draw (1.0,2.0)-- (8.0,1.0);
\draw (3.0,2.0)-- (8.0,1.0);
\draw (5.0,2.0)-- (8.0,1.0);
\draw (7.0,2.0)-- (8.0,1.0);
\draw (8.0,1.0)-- (9.0,2.0);
\draw (8.0,1.0)-- (11.0,2.0);
\draw (8.0,1.0)-- (13.0,2.0);
\draw (8.0,1.0)-- (15.0,2.0);
\draw [line width=0.9pt ,dash pattern=on 3pt off 3pt] (20.0,7.0)-- (5.0,6.0);
\draw [line width=0.9pt ,dash pattern=on 3pt off 3pt ] (19.0,6.0)-- (4.0,5.0);
\draw [line width=0.9pt,dash pattern=on 3pt off 3pt ] (18.0,5.0)-- (3.0,4.0);
\draw [line width=0.9pt,dash pattern=on 3pt off 3pt ] (17.0,4.0)-- (2.0,3.0);
\draw [line width=0.9pt,dash pattern=on 3pt off 3pt ] (6.0,7.0)-- (21.0,8.0);
\draw [line width=0.9pt ,dash pattern=on 3pt off 3pt] (1.0,2.0)-- (16.0,3.0);
\begin{scriptsize}

\end{scriptsize}
\end{tikzpicture}

\begin{tikzpicture}[line cap=round,line join=round, x=1.0cm,y=1.0cm, scale = 0.60]

\draw [line width=0.8 pt] (1.0,2.0)-- (2.0,3.0);
\draw [line width=3.6pt] (2.0,3.0)-- (3.0,2.0);
\draw [line width=0.8 pt] (4.0,3.0)-- (3.0,2.0);
\draw [line width=3.6pt] (4.0,3.0)-- (5.0,2.0);
\draw [line width=0.8 pt] (5.0,2.0)-- (6.0,3.0);
\draw [line width=3.6pt] (6.0,3.0)-- (7.0,2.0);
\draw [line width=0.8 pt] (7.0,2.0)-- (8.0,3.0);
\draw [line width=3.6pt ] (8.0,3.0)-- (9.0,2.0);
\draw [line width=0.8 pt] (2.0,3.0)-- (3.0,4.0);
\draw [line width=3.6pt ] (4.0,3.0)-- (3.0,4.0);
\draw [line width=0.8 pt] (5.0,4.0)-- (4.0,3.0);
\draw [line width=3.6pt ] (5.0,4.0)-- (6.0,3.0);
\draw [line width=0.8 pt] (6.0,3.0)-- (7.0,4.0);
\draw [line width=3.6pt ] (7.0,4.0)-- (8.0,3.0);
\draw [line width=0.8 pt] (8.0,3.0)-- (9.0,4.0);
\draw [line width=0.8 pt] (9.0,2.0)-- (10.0,3.0);
\draw [line width=0.8 pt] (3.0,4.0)-- (4.0,5.0);
\draw [line width=3.6pt ] (5.0,4.0)-- (4.0,5.0);
\draw [line width=0.8 pt] (5.0,4.0)-- (6.0,5.0);
\draw [line width=3.6pt ] (7.0,4.0)-- (6.0,5.0);
\draw [line width=0.8 pt] (7.0,4.0)-- (8.0,5.0);
\draw [line width=3.6pt ] (9.0,4.0)-- (8.0,5.0);
\draw [line width=0.8 pt] (10.0,3.0)-- (11.0,4.0);
\draw [line width=0.8 pt] (9.0,4.0)-- (10.0,5.0);
\draw [line width=3.6pt ] (10.0,5.0)-- (11.0,4.0);
\draw [line width=3.6pt ] (9.0,4.0)-- (10.0,3.0);
\draw [line width=0.8 pt] (11.0,4.0)-- (12.0,5.0);
\draw [line width=0.8 pt] (4.0,5.0)-- (5.0,6.0);
\draw [line width=3.6pt ] (6.0,5.0)-- (5.0,6.0);
\draw [line width=0.8 pt] (6.0,5.0)-- (7.0,6.0);
\draw [line width=3.6pt ] (8.0,5.0)-- (7.0,6.0);
\draw [line width=0.8 pt] (8.0,5.0)-- (9.0,6.0);
\draw [line width=3.6pt ] (10.0,5.0)-- (9.0,6.0);
\draw [line width=0.8 pt] (10.0,5.0)-- (11.0,6.0);
\draw [line width=3.6pt ] (11.0,6.0)-- (12.0,5.0);
\draw [line width=0.8 pt] (12.0,5.0)-- (13.0,6.0);
\draw [line width=3.6pt ] (11.0,4.0)-- (12.0,3.0);
\draw [line width=3.6pt ] (12.0,5.0)-- (13.0,4.0);
\draw [line width=3.6pt ] (13.0,6.0)-- (14.0,5.0);
\draw [line width=0.8 pt] (11.0,2.0)-- (12.0,3.0);
\draw [line width=0.8 pt] (12.0,3.0)-- (13.0,4.0);
\draw [line width=0.8 pt] (13.0,4.0)-- (14.0,5.0);
\draw [line width=0.8 pt] (14.0,5.0)-- (15.0,6.0);
\draw [line width=0.8 pt] (13.0,2.0)-- (14.0,3.0);
\draw [line width=0.8 pt] (15.0,4.0)-- (14.0,3.0);
\draw [line width=0.8 pt] (15.0,4.0)-- (16.0,5.0);
\draw [line width=0.8 pt] (16.0,5.0)-- (17.0,6.0);
\draw [line width=3.6pt ] (15.0,6.0)-- (16.0,5.0);
\draw [line width=3.6pt ] (14.0,5.0)-- (15.0,4.0);
\draw [line width=3.6pt ] (13.0,4.0)-- (14.0,3.0);
\draw [line width=3.6pt ] (12.0,3.0)-- (13.0,2.0);
\draw [line width=0.8 pt] (15.0,2.0)-- (16.0,3.0);
\draw [line width=0.8 pt] (17.0,4.0)-- (16.0,3.0);
\draw [line width=0.8 pt] (17.0,4.0)-- (18.0,5.0);
\draw [line width=0.8 pt] (18.0,5.0)-- (19.0,6.0);
\draw [line width=3.6pt ] (14.0,3.0)-- (15.0,2.0);
\draw [line width=3.6pt ] (15.0,4.0)-- (16.0,3.0);
\draw [line width=3.6pt ] (16.0,5.0)-- (17.0,4.0);
\draw [line width=3.6pt ] (17.0,6.0)-- (18.0,5.0);
\draw [line width=3.6pt ] (10.0,3.0)-- (11.0,2.0);
\draw [line width=0.8 pt] (5.0,6.0)-- (6.0,7.0);
\draw [line width=3.6pt ] (7.0,6.0)-- (6.0,7.0);
\draw [line width=0.8 pt] (7.0,6.0)-- (8.0,7.0);
\draw [line width=3.6pt ] (9.0,6.0)-- (8.0,7.0);
\draw [line width=0.8 pt] (9.0,6.0)-- (10.0,7.0);
\draw [line width=3.6pt ] (11.0,6.0)-- (10.0,7.0);
\draw [line width=0.8 pt] (11.0,6.0)-- (12.0,7.0);
\draw [line width=3.6pt ] (12.0,7.0)-- (13.0,6.0);
\draw [line width=0.8 pt] (13.0,6.0)-- (14.0,7.0);
\draw [line width=3.6pt ] (14.0,7.0)-- (15.0,6.0);
\draw [line width=0.8 pt] (15.0,6.0)-- (16.0,7.0);
\draw [line width=0.8 pt] (17.0,6.0)-- (18.0,7.0);
\draw [line width=3.6pt ] (16.0,7.0)-- (17.0,6.0);
\draw [line width=0.8 pt] (19.0,6.0)-- (20.0,7.0);
\draw [line width=3.6pt ] (18.0,7.0)-- (19.0,6.0);
\draw [line width=0.8 pt] (6.0,7.0)-- (7.0,8.0);
\draw [line width=3.6pt ] (8.0,7.0)-- (7.0,8.0);
\draw [line width=0.8 pt] (8.0,7.0)-- (9.0,8.0);
\draw [line width=3.6pt ] (10.0,7.0)-- (9.0,8.0);
\draw [line width=0.8 pt] (10.0,7.0)-- (11.0,8.0);
\draw [line width=3.6pt ] (12.0,7.0)-- (11.0,8.0);
\draw [line width=0.8 pt] (12.0,7.0)-- (13.0,8.0);
\draw [line width=3.6pt ] (13.0,8.0)-- (14.0,7.0);
\draw [line width=0.8 pt] (14.0,7.0)-- (15.0,8.0);
\draw [line width=3.6pt ] (15.0,8.0)-- (16.0,7.0);
\draw [line width=0.8 pt] (16.0,7.0)-- (17.0,8.0);
\draw [line width=0.8 pt] (18.0,7.0)-- (19.0,8.0);
\draw [line width=3.6pt ] (17.0,8.0)-- (18.0,7.0);
\draw [line width=0.8 pt] (20.0,7.0)-- (21.0,8.0);
\draw [line width=3.6pt ] (19.0,8.0)-- (20.0,7.0);
\draw (14.0,9.0)-- (7.0,8.0);
\draw (14.0,9.0)-- (9.0,8.0);
\draw (14.0,9.0)-- (11.0,8.0);
\draw (14.0,9.0)-- (13.0,8.0);
\draw (14.0,9.0)-- (15.0,8.0);
\draw (14.0,9.0)-- (17.0,8.0);
\draw (14.0,9.0)-- (19.0,8.0);
\draw (14.0,9.0)-- (21.0,8.0);
\draw (1.0,2.0)-- (8.0,1.0);
\draw (3.0,2.0)-- (8.0,1.0);
\draw (5.0,2.0)-- (8.0,1.0);
\draw (7.0,2.0)-- (8.0,1.0);
\draw (8.0,1.0)-- (9.0,2.0);
\draw (8.0,1.0)-- (11.0,2.0);
\draw (8.0,1.0)-- (13.0,2.0);
\draw (8.0,1.0)-- (15.0,2.0);
\draw [line width=2pt ,dash pattern=on 3pt off 3pt] (20.0,7.0)-- (5.0,6.0);
\draw [line width=2pt ,dash pattern=on 3pt off 3pt ] (19.0,6.0)-- (4.0,5.0);
\draw [line width=2pt,dash pattern=on 3pt off 3pt ] (18.0,5.0)-- (3.0,4.0);
\draw [line width=2pt,dash pattern=on 3pt off 3pt ] (17.0,4.0)-- (2.0,3.0);
\draw [line width=2pt,dash pattern=on 3pt off 3pt ] (6.0,7.0)-- (21.0,8.0);
\draw [line width=2pt ,dash pattern=on 3pt off 3pt] (1.0,2.0)-- (16.0,3.0);
\begin{scriptsize}

\end{scriptsize}
\end{tikzpicture}

\end{center}
\caption{Orthogonal chain decompositions $\{\C_i\}_{i=1}^n$ (above) and $\{\C_i'\}_{i=1}^n$ (below) of the cycle are highlighted with bold lines. Dashed lines indicate how the chains wrap around.}
\label{Ortho_decomp}
\end{figure}

Now, the two partitions we have defined have the property that if $A$ and $B$ are in $\C_i$ for some $i$, then at most one of $A$ and $B$ are in any $\C_j'$.   Moreover, since each $A$ is contained in exactly one chain in each partition, it follows that each $A$ is contained in exactly $2$ chains in the union of the two partitions.  Thus,  we have 
\begin{displaymath}
\sum_{\C \in \{\C_i\}_{i=1}^n \cup \{C_i'\}_{i=1}^n} \abs{\A \cap \C} = 2 \abs{\A}.
\end{displaymath}
On the other hand, if a chain $\C \in \{\C_i\}_{i=1}^n$ intersects $\A$ in $k> 2$ sets $A_1,A_2,\dots,A_k$ with $A_1\subset A_2 \subset \dots \subset A_k$, then there are $k-2$ chains in $\C' \in \{\C_i'\}_{i=1}^n$ such that $\abs{\A \cap \C'} =1$, namely those chains in $\{\C_i'\}_{i=1}^n$ containing $A_2,A_3,\dots, A_{k-2}$ or $A_{k-1}$,  as an intersection of greater than one would yield an induced $Y$ or $Y'$.   Similarly, if a chain $\C' \in \{\C_i'\}_{i=1}^n$ intersects $\A$ in $k > 2$ sets, then there are $k-2$ chains from $\{\C_i\}_{i=1}^n$ which intersect $\A$ in exactly one set.   
Here, we are using an additional property of the decomposition that if $A \in \C \cap \C'$, then no set larger than $A$ in $\C$ is comparable to a set larger than $A$ in $\C'$, and, similarly, no set smaller than $A$ in $\C$ is comparable to a set smaller than $A$ in $\C'$. 
We have shown that there is a total of $2k-2$ incidences of $\A$ with these $k-1$ chains. It follows that the number of pairs $(A,\C)$ where $A \in \A,  \C \in \{\C_i\}_{i=1}^n \cup \{\C_i'\}_{i=1}^n$ and $A \in \C$ is at most twice the number of chains. Thus,
\begin{displaymath}
\sum_{\C \in \{\C_i\}_{i=1}^n \cup \{C_i'\}_{i=1}^n} \abs{\A \cap \C} \le 2 \abs{ \{\C_i\}_{i=1}^n \cup \{C_i'\}_{i=1}^n} = 4n.
\end{displaymath}
Dividing through by $2$ yields the desired inequality.  
\end{proof}

Lemma \ref{ortho} implies the LYM-type inequality, Theorem \ref{inducedLYM}, exactly as in the previous proofs.  It remains to derive the bound on $\La^\#(n,Y,Y')$ using Theorem \ref{inducedLYM}.

\begin{proof}[Proof of Theorem \ref{induced}]
 If $\A$ contains neither $\varnothing$ nor $[n]$, then we are done by Theorem \ref{inducedLYM}.  If $\varnothing$ and $[n]$ are in $\A$, then $\A \setminus \{\varnothing,[n]\}$ is induced $V$ and $\Lambda$ free.  It follows from  Katona and Tarjan \cite{katona1983extremal} that
 \begin{displaymath}
  \La^\#(n,Y,Y') \le 2 + \La^\#(n,V,\Lambda) = 2 + 2\binom{n-1}{\floor{\frac{n-1}{2}}} \le \Sigma(n,2).
  \end{displaymath}
  Now, assume without loss of generality that $\varnothing \not\in \A$ but $[n] \in \A$, and let $\A' = \A \setminus \{[n]\}$.  Since $\A'$ is induced $Y$ and $Y'$-free,  it satisfies the hypothesis of Theorem \ref{inducedLYM}. Assume, by contradiction, that $\abs{\A'} = \Sigma(n,2)$. It follows that equality holds in Theorem \ref{inducedLYM}. If $n$ is odd, then we have that $\A'  = \binom{[n]}{\floor{n/2}} \cup \binom{[n]}{\ceil{n/2}}$ which implies $\A$ induces a $Y'$, contradiction.  If $n$ is even, then $\binom{[n]}{n/2} \subset \A'$ and $\binom{n}{n/2+1}$ sets from $\binom{[n]}{n/2-1} \cup \binom{[n]}{n/2+1}$  are in $\A$.  Since $\A$ contains no $Y'$, it follows that $\A' \cap \binom{[n]}{n/2+1} = \varnothing$.  Thus, we must have $\A' = \binom{[n]}{n/2 -1} \cup \binom{[n]}{n/2}$, but then $\A$ still contains an induced $Y'$, contradiction.

\end{proof}

\subsection*{Acknowledgements}
We thank D\"{o}m\"{o}t\"{o}r P\'{a}lv\"{o}lgyi for many helpful discussions which were especially useful to the proof of Lemma \ref{ortho}. We also thank D\'{a}niel Gerbner, Guillaume Lagarde, Gyula O.H.Katona, D\'aniel Nagy and Bal\'azs Patk\'os for many useful discussions and Anna Jablonkai for helping with the illustrations.

\end{document}